\documentclass[]{amsart}
\usepackage{amsthm}
\usepackage{amssymb}
\usepackage{enumerate}
\usepackage{mathtools}

\newcommand{\R}{\mathbb{R}}
\newcommand{\Z}{\mathbb{Z}}
\newcommand{\N}{\mathbb{N}}
\newcommand{\B}{\mathbb{B}}

\renewcommand{\le}{\leqslant}
\renewcommand{\ge}{\geqslant}

\theoremstyle{plain}% default
\newtheorem{thm}{Theorem}[section]
\newtheorem{lem}[thm]{Lemma}

\newtheorem{thmR}{Theorem}

\newtheorem{quR}{Question}

\makeatletter
\newtheorem*{rep@theorem}{\rep@title}
\newcommand{\newreptheorem}[2]{%
\newenvironment{rep#1}[1]{%
 \def\rep@title{#2 \ref{##1}}%
 \begin{rep@theorem}}%
 {\end{rep@theorem}}}
\makeatother
\newreptheorem{theorem}{Theorem}

\newtheorem*{thm*}{Theorem}
\newtheorem*{lem*}{Lemma}
\newtheorem*{prop*}{Proposition}
\newtheorem*{cor*}{Corollary}
\newtheorem*{qu*}{Question}
\newtheorem*{dt*}{Definition and Theorem}
\newtheorem*{not*}{Notation}
\newtheorem*{exmp*}{Example}
\newtheorem*{exmps*}{Examples}
\newtheorem*{dprop*}{Definition and Proposition}
\newtheorem*{conj*}{Conjecture}

\theoremstyle{definition}
\newtheorem{defn}[thm]{Definition}

\newtheorem*{defn*}{Definition}

\theoremstyle{plain}
\newtheorem{rem}[thm]{Remark}
\newtheorem*{rem*}{Remark}

\DeclareMathOperator\Supp{supp}
\DeclareMathOperator\dc{dc}
\DeclareMathOperator\ord{ord}
\DeclareMathOperator\Exp{Exp}

\DeclarePairedDelimiter\floor{\lfloor}{\rfloor}
\DeclarePairedDelimiter\ceil{\lceil}{\rceil}

\begin{document}
\title{The degree of commutativity and lamplighter groups}
\author{Charles Garnet Cox}
\address{Mathematical Sciences, University of Southampton, SO17 1BJ, UK}
\email{cpgcox@gmail.com}
\thanks{}

\subjclass[2010]{20P05}

\keywords{Wreath products, lamplighter group, degree of commutativity, exponential growth}

\date{August 23, 2018}
\begin{abstract}
The degree of commutativity of a group $G$ measures the probability of choosing two elements in $G$ which commute. There are many results studying this for finite groups. In \cite{dcA}, this was generalised to infinite groups. In this note, we compute the degree of commutativity for wreath products of the form $\Z\wr \Z$ and $F\wr \Z$ where $F$ is any finite group.
\end{abstract}
\maketitle
\section{Introduction}
Let $F$ be a finite group. Then the degree of commutativity of $F$, denoted $\dc(F)$, is the probability of choosing two elements in $F$ which commute i.e.\
\[\dc(F):=\frac{|\{(a, b) \in F^2 :  ab=ba\}|}{|F|^2}.\]

This definition was generalised to infinite groups in \cite{dcA} in the following way. Let $G$ be a finitely generated group and $S$ be a finite generating set for $G$. Let $|g|_S$ denote the length of $g$ with respect to the generating set $S$ i.e.\ the infimum of all word lengths of words in $S$ which represent $g$. For any $n \in \N$, let the ball of radius $n$ in the Cayley graph of $G$ with respect to the generating set $S$ be denoted by $\B_S(n)$. Thus $\B_S(n)=\{g \in G : |g|_S\le n\}$. Then the degree of commutativity of $G$ with respect to $S$, as defined in \cite{dcA}, is
\begin{align}\label{dcdefn}
\limsup_{n\rightarrow\infty}\frac{|\{(a, b) \in \B_S(n)^2 :  ab=ba\}|}{|\B_S(n)|^2}
\end{align}
and is denote by $\dc_S(G)$. They also pose an intriguing conjecture.
\begin{conj*}\cite[Conj. 1.6]{dcA} Let $G$ be a finitely generated group, and let $S$ be a finite generating set
for $G$. Then: (i) $\dc_S(G)>0$ if and only if $G$ is virtually abelian; and (ii) $\dc_S(G) > 5/8$
if and only if $G$ is abelian.
\end{conj*}
They verify this conjecture for hyperbolic groups and groups of polynomial growth (see \cite{introtogrowth} for an introduction to the growth of groups). In this note we will investigate the conjecture for groups which are wreath products.

Perhaps the best known examples of infinite wreath products are the lamplighter groups $C\wr \Z$ where $C$ is cyclic. Such groups are sensible to investigate with respect to the conjecture since they have exponential growth and yet all elements in the base of $C\wr \Z$ commute. We obtain the following.

\begin{thmR} \label{mainthm} Let $G=C\wr\Z$ where $C$ is a non-trivial cyclic group. Then there is a generating set $S$ of $G$ such that $\dc_S(G)=0$.
\end{thmR}

This work generalises to allow us to replace `cyclic' with `finite'.
\begin{thmR}\label{mainthm2}Let $G:=F\wr \Z$ where $F$ is a non-trivial finite group. Then there is a generating set $S$ of $G$ such that $\dc_S(G)=0$.\end{thmR}

Note that the groups of Theorem \ref{mainthm2} include the first known examples of non-residually finite groups with degree of commutativity 0, since it is currently open as to whether there exists a non-residually finite hyperbolic group. 

\begin{rem*} In the case where $G$ is finite, it is well known that
 \[\dc(G)=\frac{\#\;\text{conjugacy classes of}\;G}{|G|}.\]
 One could therefore define the degree of commutativity for any finitely generated infinite group with respect to a finite generating set $S$ to be
 \[\limsup_{n\rightarrow \infty}\frac{\#\;\text{conjugacy classes intersecting}\;\B_S(n)}{|\B_S(n)|}.\]
 Such a limit may not be a real limit. Note that this definition includes the conjugacy growth function of $G$, which was introduced in \cite{introtoconjgrowth} and studied, for example, in \cite{sapirconjgrowth} and \cite{osinconjgrowth}.
\end{rem*}
Two questions then present themselves.
\begin{quR} With this definition for degree of commutativity, does the conjecture above (from \cite{dcA}) hold?
\end{quR}
\begin{quR} Does this definition for the degree of commutativity coincide with (\ref{dcdefn}) above?
\end{quR}
The author is unaware of such questions being posed before, and these questions are not discussed further in this note.

\vspace{0.25cm}
\noindent\textbf{Acknowledgements.} This work would not have been completed without the guidance of my PhD supervisor, Armando Martino. I also thank the other authors of \cite{dcA} for a paper filled with so many ideas. Finally, I thank the referee for their helpful comments, and for pointing out \cite{lamplighternormalform} which streamlines the main arguments.
\vspace{0.25cm}

We now introduce wreath products from an algebraic viewpoint, but will provide intuition (using permutations) below.
\begin{defn*} Given groups $G$ and $H$, the \emph{unrestricted wreath product} of $G$ and $H$ has elements consisting of an element $h \in H$ and a function $f: H\rightarrow G$. Let $B'$ be the set of all such functions. If $f_1, f_2 \in B'$, then $(f_1\times f_2)(h):=f_1(h)\cdot f_2(h)$
for all $h \in H$, where $\cdot$ denotes the binary operation of $G$. Moreover if $k \in H$ then $k^{-1}(f(h))k:=f(hk^{-1})$ for all $h \in H$. This is equal to the semidirect product $B'\rtimes H$. The \emph{restricted wreath product}, denoted $G\wr H$, is defined analogously as the semidirect product $B\rtimes H$ where $H$ is the \emph{head} of $G\wr H$ and $B$, the \emph{base} of $G\wr H$, is the subgroup of $B'$ consisting of functions with finite support i.e.\ functions $f \in B'$ such that $f(h)\ne1$ for only finitely many $h$. Since the base is a direct sum of $|H|$ copies of $G$, for any $h \in H$ let $G_h$ denote the copy of $G$ corresponding to $h$.
\end{defn*}
It may be useful to provide some of the intuition used when thinking about lamplighter groups i.e.\ groups of the form $C\wr \Z$ where $C$ is cyclic. Each of these groups acts naturally on the corresponding set $C\times \Z$. We shall picture $C$ as addition modulo $n$ if $|C|=n$ and as $\Z$ otherwise. Hence $C=\{0, 1, \ldots, n-1\}$ or $C=\Z$. A well used generating set is $\{a_0, t\}$ where $\Supp(a_0)=\{(0,0), (1,0), \ldots, (n-1, 0)\}$ and $\Supp(t)=C\times \Z$ with $t:(m,n)\rightarrow (m,n+1)$ for all $m \in C$ and $n \in \Z$. In the case where $|C|=2$, the base of $C\wr \Z$ can be thought of as a countable collection of street lamps, with each lamp having an `off' or `on' setting. If $2<|C|<\infty$, then we can consider each `lamp' to have a finite number of settings (possibly corresponding to different levels of brightness). In the case of $\Z\wr \Z$, the base can be thought of as lamps, where each lamp has an associated `voltage' which takes a value in $\Z$. Although this intuition will not be taken any further, it can also be seen to apply to subgroups of $\R \wr \R$.

\begin{rem*} Throughout this paper we work with fixed generating sets. This is because at the present it is unknown whether a change of generating set affects $\dc$ (as defined in (\ref{dcdefn})) and the negligibility of a set (as defined in Lemma \ref{mainlem} below). For a group $G$ finitely generated by $X$ and $Y$ we can say that the metrics on the Cayley graphs produced by $X$ and $Y$ are Bi-Lipschitz equivalent. This also means that there is a constant $d\in \N$ such that $|\B_Y(n/d)|\le|\B_X(n)|\le|\B_Y(dn)|$ for all $n\in \N$.
\end{rem*}

\section{Proving Theorem \ref{mainthm}}
The key result we shall draw upon is the following. For the group $G=H\wr \Z$ we shall use the base of $H\wr \Z$ as the set $N$.
\begin{lem}\cite[Lem. 3.1]{dcA}\label{mainlem} Let $G$ be a finitely generated group, and let $S$ be a finite generating system
for $G$. Suppose that there exists a subset $N \subseteq G$ satisfying the following conditions:
\begin{enumerate}[i)]
 \item $N$ is $S$-negligible, i.e.\, $\lim_{n\rightarrow\infty}\frac{|N \cap\B_S(n)|}{|\B_S(n)|}=0$;
 \item $\lim_{n\rightarrow\infty}\frac{|C_G(g) \cap \B_S(n)|}{|\B_S(n)|}=0$ uniformly in $g \in G \setminus N$.
\end{enumerate}
Then, $\dc_S(G) = 0$.
\end{lem}

\begin{rem*}
Throughout we will restrict ourselves to generating sets which are the union of a generator of $\Z$ and a generating set for $H_i$ for some fixed $i \in \Z$.
\end{rem*}

\subsection{Proving that groups $C\wr\Z$ satisfy (ii) of Lemma \ref{mainlem}}\label{conditionii}
This is the simpler of the two conditions to prove for such groups. We first introduce the translation length of a group. For more discussions on these, see \cite{connertranslationnumbers} and the references therein.
\begin{defn} Let $G$ be a finitely generated group with finite generating set $S$ and let $g \in G$. Then $\tau_S(g):=\limsup_{n\rightarrow \infty}\frac{|g^n|_{_S}}{n}$ is the \emph{translation length} of $g$. Let $F(G)$ denote the set of non-torsion elements in $G$. If there is a finite generating set $S'$ of $G$ such that $\{\tau_{S'}(g):g \in F(G)\}$ is uniformly bounded away from 0, then we say that $G$ is \emph{translation discrete}. If a group is translation discrete with respect to one finite generating set, it is translation discrete with respect to all generating sets (see \cite[Lem. 2.6.1]{translationdiscreteproof}).
\end{defn}
Note that $|g^n|_S/n\ge \tau_S(g)$ for all $n>0$. We shall use the following.
\begin{lem} Let $G$ be finitely generated, $S$ a finite generating set for $G$, and $|\B_S(n)|\ge f(n)$ for all $n \in \N$, where $f$ is a polynomial of degree 2. Let $N\subseteq G$. If (i) $C_G(g)$ is cyclic for all $g \in G\setminus N$; and (ii) the translation lengths of the elements in $T:=\{h \in F(G)\mid h\in C_G(g)$ for some $g\in G\setminus N\}$ are uniformly bounded away from 0, then $\lim_{n\rightarrow\infty}\frac{|C_G(g) \cap \B_S(n)|}{|\B_S(n)|}=0$ uniformly in $g \in G \setminus N$.
\end{lem}
\begin{proof} 
This argument can be found within the proof of \cite[Thm. 1.7]{dcA}. From (ii), there exists a constant $\lambda \in \N$ such that $\tau_S(h)\ge 1/\lambda$ for all $h \in T$.

Let $g \in G\setminus N$. By (i), $C_G(g)=\langle h\rangle$ for some $h \in G$. We now consider how $C_G(g)\cap \B_S(n)$ grows with respect to $n$. If $h$ is torsion then there is nothing to prove. Otherwise $h \in T$. Then $h^k \in C_G(g)\cap \B_S(n)$ means $|h^k|_{_S}\le n$ so that $|h^k|_{_S}\ge |k|\tau_S(h)\ge |k|/\lambda$. Thus $|k|\le \lambda n$ and $|C_G(g)\cap \B_S(n)|\le 2\lambda n+1$. Finally, since $\B_S(n)$ grows faster than any linear function, the claim follows.
\end{proof}
We must therefore show the two conditions in this lemma are satisfied. Note that they are independent of the choice of finite generating set used.
\begin{defn} Let $A$ denote the base of $G=H\wr \Z$ where $H$ is a finitely generated group. If $g \in A\setminus \{1\}$, then $g=\prod_{i \in I}g_i$ where $I$ is a finite subset of $\Z$ and $g_i \in H_i\setminus\{1\}$ for each $i \in I$. Now $g_{\min}:=\inf\{I\}$ and $g_{\max}:=\sup\{I\}$, the infimum and supremum of $I$, respectively. 
\end{defn}
\begin{lem} Let $G:=H\wr \Z$ and let $A$ denote the base of $G$. If $g \in A$, then $C_G(g)\le A$ (and if $H$ is abelian, then $C_G(g)=A$). If $g \in G\setminus A$, then $C_G(g)$ is cyclic (and $C_G(g)$ contains no non-trivial element of the base).
\end{lem}
\begin{proof} The first claim is clear. For the second, let $g \in G\setminus A$, so that $g=wt^k$ for some $w \in A$ and $k \in \Z\setminus \{0\}$. Now, for any $v \in A$,
\begin{align}\label{eqn1-wreathcentralisers}
 &v^{-1}wt^kv=wt^k\nonumber\\
 \Leftrightarrow&v^{-1}wt^kvt^{-k}=w\nonumber\\
 \Leftrightarrow&t^kvt^{-k}=w^{-1}vw
\end{align}
and so, if $v$ is non-trivial, then $(w^{-1}vw)_{\min}>(t^kvt^{-k})_{\min}$ and so $v \not\in C_G(wt^k)$. 
Now assume that $vt^{\alpha}\in C_G(wt^k)$. If $v't^{\alpha} \in C_G(wt^k)$, then $v't^{\alpha}(vt^{\alpha})^{-1}=v'v^{-1}$ and so by (\ref{eqn1-wreathcentralisers}), $v'v^{-1}=1$ i.e.\ $v'=v$. Thus for each $s \in \Z$ such that $vt^s \in C_G(wt^k)$ there is no $v'\ne v$ such that $v't^s \in C_G(wt^k)$. Now assume that $\alpha$ is the smallest positive integer such that there exists a $v \in A$ with $vt^{\alpha} \in C_G(wt^k)$. If, for some $\beta \in \Z$ there is a $u \in A$ such that $ut^\beta \in C_G(wt^k)$, then, by the division algorithm, $\beta=n\alpha$ for some $n \in \Z$. Thus $ut^\beta=(vt^{\alpha})^n$ since for each $s \in \Z$ there is at most one $\nu \in A$ such that $\nu t^s \in C_G(wt^k)$.
\end{proof}

\begin{lem} Let $G=H\wr \Z$ where $H$ is a finitely generated group and let $A$ denote the base of $G$. Then $\{\tau_S(g) :  g \in G\setminus A\}$ is uniformly bounded away from 0 i.e.\ if $H$ is torsion, then $G$ is translation discrete.
\end{lem}
\begin{proof} Let $S_H$ denote a finite generating set for $H_0$. We work with the generating set $S:=S_H\cup\{t\}$ of $G$.

If $g \in G\setminus A$, then $g=wt^k$ where $w \in A$ and $k \in \Z\setminus\{0\}$. Thus for any $n \in \N$, $|g^n|_{_S}\ge |k|n\ge n$ and so $\tau_S(g)\ge 1$. 
\end{proof}
Let $H$ be finitely generated with $\tau_S(H)\subseteq\N\cup \{0\}$ for some finite generating set $S$. Then one can prove, with $S'$ as a finite generating set consisting of the generating set $S$ for $H_0$ and a generator of the head of $H\wr \Z$, that $\tau_{S'}(H\wr \Z)=\N\cup\{0\}$ and that $\tau_{S'}^{-1}(0)$ is equal to $\{w \in \bigoplus_{i \in I}H_i\mid I$ is a finite subset of $\Z$ and $w$ is torsion$\}$. Moreover, if we drop the condition on the translation lengths of $H$ and let $A$ denote the base of $H\wr \Z$, then $\tau_{S'}(H\wr\Z\setminus A)=\N$. 

\subsection{Proving that groups $C\wr\Z$ satisfy (i) of Lemma \ref{mainlem}}\label{conditioni}
For a group $G$ with finite generating set $S$, the exponential growth rate of $G$ with respect to $S$ is
\begin{align*}
\Exp_S(G) := \lim_{n\rightarrow\infty} \sqrt[n]{ |\B_S(n)|}.
\end{align*}

\begin{defn} A group $G$ with finite generating set $S$ is said to have exponential growth if $\Exp_S(G) > 1$, and subexponential growth if $\Exp_S(G) = 1$. This does not depend on the choice of finite generating set.
\end{defn}

The author is unaware of how to show that the negligibility of a set is independent of the generating set used. When working with groups of exponential growth, the `density' of a set $A\subset G$ can depend on the choice of generating set \cite[Example 1.5]{density}. Here, density of a subset $A$ of $G=\langle S\rangle$ is thought of as the number
\[\limsup_{n\rightarrow\infty}\frac{|A \cap\B_S(n)|}{|\B_S(n)|}\]
so that a set is negligible if and only if it has density 0. Note that if the negligibility of a set is independent of the finite generating set used, then the results that follow would apply to any finite generating set.

\begin{rem*}
We shall work with the generating set $\langle a_0, t\rangle$ where $a_0$ is a generator of $C_0$ and $t$ is a generator of the head. The arguments also work for $C_i$ for any $i \in \Z$.
\end{rem*}
The following is a simplification of \cite[Prop. 3.8]{lamplighternormalform}, and differs by applying only to elements of the base rather than all elements of the group.

\begin{lem}\label{formofwords} Let $H$ be finitely generated by $X$, let $A$ denote the base of $H\wr \Z$, and let $g\in A\setminus\{1\}$. Then $g=\prod_{i\in I}g_i$ where $I=\{i_1, \ldots, i_k, j_1, \ldots, j_l\}\subset \Z$, $g_i\ne 1$ for each $i \in I$, and w.l.o.g. $j_l<\ldots<j_1<0\le i_1<i_2<\ldots <i_k$. For every $i \in I$, let $w^{(i)}$ be a representative of $t^ig_it^{-i}$ of minimal length. Then
\begin{align}\label{theword}
t^{-j_l}w^{(j_l)}t^{j_l-j_{l-1}}w^{(j_{l-1})}t^{j_{l-1}-j_{l-2}}\ldots w^{(j_1)}t^{j_1-i_k}w^{(i_k)}t^{i_k-i_{k-1}}w^{(i_{k-1})}\ldots w^{(i_1)}t^{i_1}
\end{align}
is a word in $X_0\cup\{t\}$ that obtains the minimal length amongst all representatives of $g$ (called the left first form of $g$).
\end{lem}

If the $g$ in Lemma \ref{formofwords} had $g_{\min}<0$, consider $t^{j_l}gt^{-j_l}$. We note that $|t^{j_l}gt^{-j_l}|_{X_0\cup\{t\}}\\\le |g|_{X_0\cup\{t\}}$ by conjugating (\ref{theword}) by $t^{j_l}$.

\begin{rem}\label{onlypositivelamps} Let $A_s:=\{g \in A :  g_{\min}\ge s\}$. By the previous paragraph, for any $s \in \Z\setminus \N$, $|\B_S(n)\cap (A_s\setminus A_{s+1})|\le|\B_S(n)\cap A_0|$. Combining this with the fact that for any $s\le -n$, $|\B_S(n)\cap (A_s\setminus A_{s+1})|=0$, we see that $|\B_S(n)\cap A|\le (n+1)|\B_S(n)\cap A_0|$.
\end{rem}

We are now ready to prove the first case of Theorem 1.

\begin{lem} \label{firstcase} Let $G=C_2\wr \Z$, $S=\{a_0, t\}$, and $A$ denote the base of $G$. Then $A$ is $S$-negligible.
\end{lem}
\begin{proof} Fix an $n \in \N$. Our aim is to produce a bound for $|\B_S(n)\cap A|$. From \cite{growthofwreathproducts}, the exponential growth rate of $G$ with respect to $S$ is $\frac{1+\sqrt5}{2}$. By Remark \ref{onlypositivelamps}, it is sufficient to bound the exponential growth rate of $|\B_S(n)\cap A_0|$ (elements of the base with $g_{\min}\ge0$). Moreover we may work with elements of exactly length $n$, since $\B_S(n)$ and $\B_S(n)\setminus\B_S(n-1)$ have the same exponential growth rate. Therefore we will show that $|(\B_S(n)\setminus \B_S(n-1))\cap A_0|$ has exponential growth rate bounded by $\sqrt2$ (since $\sqrt2<\frac{1+\sqrt5}{2}$).

Let $g \in A$, $|g|_S=n$, and $g_{\min}\ge0$. For all $i>\floor*{n/2}$ the $g_i$ are trivial by (\ref{theword}). Thus our conditions on $g$ imply that $g=\prod_{i\in I}a_i$ and $I\subseteq\{0,\ldots, \floor*{n/2}\}$. By Lemma \ref{formofwords}, $g$ can be represented by a word of the form
\begin{align}\label{elementsofthebase} 
t^{-k}w_0tw_1t\ldots tw_k
\end{align}
for some $k\in \{0, 1\ldots, \floor*{n/2}\}$ and words $w_i$ which, for each $i\in\{0,\ldots,k\}$ are either the empty word or are equal to $a_0$. Note that, since $|g|_S=n$,
\begin{align*}
\sum\limits_{i=0}^{k}|w_i|_{\{a_0\}}= n-2k.
\end{align*}
Thus, for each $k$, there are at most $2^{k+1}$ options for the values of $\{w_i : i=0,1,\ldots, k\}$. Hence the size of $|\B_S(n)\cap A_0|$ is bounded by
\[\sum\limits_{j=0}^{\floor*{n/2}}2^{j+1}\le 4\cdot (\sqrt2)^n\]
and, since $\sqrt2<\frac{1+\sqrt5}{2}$, the base of $C_2\wr \Z$ is negligible.
\end{proof}
In order to prove Theorem 1, all that is required is to generalise the above lemma to the case $G=C\wr\Z$, where $|C|>2$. Our approach will be similar, but the word (\ref{elementsofthebase}) will have $w^{(i)}=a_0^{d_i}$ for some numbers $d_i\in\Z$. In order to produce a bound for the number of such words, we will use known results regarding the number of possible compositions of a number.
\begin{defn} A multiset, denoted $[\ldots]$, is a collection of objects where repeats are allowed e.g.\ $[1,2,2,3,5]$. An ordered multiset, denoted $[\ldots]_{\ord}$, is a multiset with a given ordering. Thus $[1,2,2,3,5]_{\ord}\ne[1,2,3,2,5]_{\ord}$.
\end{defn}

\begin{defn} Let $n \in \N$. Then a composition of $n$ is an ordered collection of natural numbers that sum to $n$. Thus there is a natural correspondence between compositions of $n$ and ordered multisets whose elements lie in $\N$ and sum to $n$. A weak composition of $n$ is an ordered collection of non-negative integers that sum to $n$. There is a natural correspondence between weak compositions of $n$ and ordered multisets whose elements lie in $\N\cup \{0\}$ and sum to $n$.
\end{defn}
The following are well known. See, for example, \cite{Riordan}.
\begin{lem} 
Let $n \in \N$. Then the number of compositions of $n$ is $2^{n-1}$. 
\end{lem}

\begin{lem}\label{combinatorialresult} Let $n \in \N$. Then the number of weak compositions of $n$ into exactly $k$ parts is give by the binomial coefficient
\[\begin{pmatrix}
   n+k-1\\k-1
  \end{pmatrix}.\]
\end{lem}

\begin{reptheorem}{mainthm}Let $G=C\wr\Z$ where $C$ is a non-trivial cyclic group. Let $S:=\langle a, t\rangle$ be a generating set for $G$ with $a \in C_i$ (for some $i \in \Z$) and $t \in \Z$, the head of $G$. Then $\dc_S(G)=0$.
\end{reptheorem}

\begin{proof} We will show that the base $A$ of $G$ is negligible in $G$ for the case where $|C|>2$. As with the proof of Lemma \ref{firstcase}, Remark \ref{onlypositivelamps} implies that a bound on the exponential growth rate of $|\B_S(n)\cap A_0|$ (elements of the base with $g_{\min}\ge0$) is sufficient to bound the exponential growth rate of $|\B_S(n)\cap A|$, and both $\B_S(n)$ and $\B_S(n)\setminus\B_S(n-1)$ have the same exponential growth rate. Our aim is therefore to bound the exponential growth rate of $|(\B_S(n)\setminus\B_S(n-1))\cap A_0|$.

Fix an $n \in \N$. Let $g \in A$, $|g|_S=n$, and $g_{\min}\ge0$. By Lemma \ref{formofwords}, there is a word of length $n$ of the form
\begin{align}\label{elementsofthebase2}
t^{-k}w_0tw_1t\ldots tw_k
\end{align}
where $k\in \{0, 1,\ldots, \floor*{n/2}\}$ and for each $i \in \{0,1,\ldots,k\}$ we have that $w_i=a^{d_i}$ for some $d_i \in \Z$. Since $|g|_S=n$,
\begin{align*}
\sum\limits_{i=0}^{k}|w_i|_{\{a\}}= n-2k.
\end{align*}
From \cite{growthofwreathproducts}, the growth of $C\wr \Z$ with our generating set has exponential growth rate bigger than 2 if $|C|\ge3$. 

We now use Lemma \ref{combinatorialresult}. Our aim is to show that $|(\B_S(n)\setminus \B_S(n-1))\cap A_0|$ is bounded by a function with exponential growth rate 2. Fix a $k \in \{0, 1, \ldots, \floor*{\frac{n}{2}}\}$. Each element of $|(\B_S(n)\setminus \B_S(n-1))\cap A_0|$ can be represented by a word of the form (\ref{elementsofthebase2}), where $\sum_{i=0}^k|w_i|=n-2k$. For each $i \in \{0,\ldots,k\}$ we can encode the element $w_i$ using a pair $(u^{(i)}, v^{(i)})$: for $w_i\ge0$ let $u^{(i)}:=d_i$ and $v^{(i)}:=0$ whereas for $w_i<0$ let $u^{(i)}:=0$ and $v^{(i)}:=|d_i|$. Thus, for each $i\in \{0, \ldots, k\}$, we have that $u^{(i)}, v^{(i)} \in \N\cup \{0\}$ and that $u^{(i)}v^{(i)}=0$. Each word is then in bijection with an ordered multiset
\begin{align}\label{elementsofthebase3}
[u^{(0)}, v^{(0)}, u^{(1)}, v^{(1)},\ldots, u^{(k-1)}, v^{(k-1)}, u^{(k)}, v^{(k)}]_{\ord}.
\end{align}
Each such multiset corresponds to a weak composition of $n-2k$ into $2k+2$ parts. Lemma \ref{combinatorialresult} states that there are
\begin{align*}
\begin{pmatrix}
                          n-2k+2k+2-1\\2k+2-1
                         \end{pmatrix}=\begin{pmatrix}
                          n+1\\2k+1
                         \end{pmatrix}
\end{align*}
weak compositions of $n-2k$ into $2k+2$ parts. Now we sum over all viable $k$:
\begin{align*}
\sum\limits_{k=0}^{\floor*{\frac{n}{2}}}\begin{pmatrix}
                                                 n+1\\2k+1
                                                \end{pmatrix}
&\le \sum\limits_{j=0}^{n+1}\begin{pmatrix}
                                                   n+1\\j
                                                  \end{pmatrix}=2^{n+1}.
\end{align*}
Hence $|(\B_S(n)\setminus \B_S(n-1))\cap A|\le (n+1)|(\B_S(n)\setminus \B_S(n-1))\cap A_0|\le (n+1)\cdot2^{n+1}$, and so is negligible in $C\wr \Z$.
\end{proof}

\begin{rem}\label{C2C2} Consider the group $(C_2\times C_2)\wr \Z$ with the generating set $X$ consisting of two non-trivial elements $a, b$ of $(C_2\times C_2)$ and $t'$, a generator of the head. This has the same exponential growth rate (using our generating set $\{a_0, t\}$) as $C_4\wr \Z$. Moreover the count for elements of the base of $C_4\wr \Z$ in the proof of Theorem \ref{mainthm} also applies to $(C_2\times C_2)\wr \Z$ with the generating set $X$. This can be seen by the map of sets defined by $a\mapsto a_0, b\mapsto a_0^{-1}, ab\mapsto a_0^2, t'\mapsto t$. Therefore the base of $(C_2\times C_2)\wr \Z$ is negligible with respect to $X$.\end{rem}

We now generalise the previous proof to apply to $F\wr\Z$, where $F$ is any finite group, with respect to specifically chosen generating sets.

\begin{reptheorem}{mainthm2}
Let $G:=F\wr \Z$ where $F$ is a non-trivial finite group. Then there is a generating set $S$ of $G$ such that $\dc_S(G)=0$.\end{reptheorem}

\begin{proof} Let $|F|=m>1$ and let $A$ denote the base of $G$. Then $A:=\bigoplus_{i \in \Z}F_i$ where $F_i=F$ for each $i \in \Z$. Let $S$ denote the generating set consisting of the non-trivial elements of $F_0$ and a generator $t$ of the head of $G$. From Section \ref{conditionii} we need only show that $A$ is negligible in $G$.

First we produce a lower bound on the growth of $G$. Consider words of the form
\begin{align*}
w_1tw_2tw_3\ldots tw_kt^{\epsilon}
\end{align*}
where $w_i \in S$ for each $i \in \{1,\ldots, k\}$, and $\epsilon \in \{0, 1\}$ depending on whether the word should be of odd or even length respectively. There are $m^k$ such words (since $|S|=m$). We now show that each word represents a distinct element of $G$. If two words have different $t$ exponent sums, then they have different images when we quotient by $A$. Now, by multiplying the elements on the left by $t^{-s}$, where $s$ is the exponent sum of each word, they become elements of $A$. The action of such a word is then clear: if $w_i = t$, then $w_it$ prints `blank' across the next two copies of $F$; whereas if $w_i\in F$, then $w_it$ can be thought of as `printing' the element $w_i$ on the current copy of $F$ and then moving the `head' across once. Thus each word represents a different set of instructions depending on the choices of $\{w_i\;;\;i=1,\ldots,k\}$. Thus $|\B_S(n)|\ge |\B_S(n)\setminus \B_S(n-1)|\ge m^{\ceil*{n/2}}$.

We now produce an upper bound on the exponential growth rate of $A$. As with the previous proof, we produce an upper bound for words $g \in  A\cap(\B_S(n)\setminus\B_S(n-1))$ with $g_{\min}\ge 0$. By Lemma \ref{formofwords}, each such word has a representative of the form
\begin{align}\label{thewordsthm2}
t^{-k}w_0tw_1t\ldots w_{k-1}tw_k
\end{align}
where each $w_i$ is either trivial or in $S\setminus\{t\}$. There must be at least one non-trivial $w_i$ from our hypothesis that $g \in  A\cap(\B_S(n)\setminus\B_S(n-1))$. Therefore if $n=2m$, then $k\le m-1$, and if $n=2m+1$, then $k\le m$. Hence $k\le \floor*{\frac{n-1}{2}}$. Similarly if $n=3m$, then $k\ge m$; if $n=3m+1$, then $k\ge m$; and if $n=3m+2$, then $k\ge m+1$. Hence $\floor*{\frac{n+1}{3}}\le k \le \floor*{\frac{n-1}{2}}$. By the same process the number of $i$ such that $w_i$ is non-trivial must be between 1 and $\floor*{\frac{n-1}{3}}$ . Let $d$ denote the number of non-trivial $w_i$. Then $d\in \{1,\ldots, \floor*{\frac{n-1}{3}}\}$. From the fact we are in $A_0\cap(\B_S(n)\setminus\B_S(n-1))$ and the form of (\ref{thewordsthm2}), we have that $n-2k=d$, that there are at most $\binom{k+1}{d}$ options for the positions of the non-trivial $w_i$, and that there are $m-1$ possibilities for each non-trivial $w_i$. Noting that $k+1\le\floor*{\frac{n+1}{2}}$ allows us to produce an upper bound

\begin{align*}
\sum\limits_{k=\floor*{\frac{n+1}{3}}}^{\floor*{\frac{n-1}{2}}}\begin{pmatrix}\floor*{\frac{n+1}{2}}\\n-2k\end{pmatrix}(m-1)^{n-2k}
&<(m-1)^{\ceil*{\frac{n}{3}}}\sum\limits_{k=0}^{\floor*{\frac{n-1}{2}}}\begin{pmatrix}\floor*{\frac{n+1}{2}}\\k\end{pmatrix}
\\&<(m-1)^{\ceil*{\frac{n}{3}}}\cdot 2^{\floor*{\frac{n+1}{2}}}
\\&<(m-1)\cdot(m-1)^{\frac{n}{3}}\cdot 2\cdot 2^{\frac{n}{2}}
\\&<2\cdot(m-1)\cdot\left(2(m-1)^{\frac23}\right)^{\frac{n}{2}}.
\end{align*}
If $m\ge6$, then $m>2(m-1)^{2/3}$, and the exponential growth rate of $\B_S(n)$ is greater than $A_0\cap(\B_S(n)\setminus\B_S(n-1))$. This proves $\dc_S(F\wr\Z)=0$ for $|F|\ge 6$. If $|F|<6$, then $F$ is either cyclic, and so dealt with by Theorem \ref{mainthm}, or $F$ is $C_2\times C_2$, and so dealt with by Remark \ref{C2C2}.
\end{proof}

We end by posing two questions, both of which could represent future work. These seem natural in the context of Theorem \ref{mainthm} and Theorem \ref{mainthm2}. 
\begin{quR} To what extent can the approach used above apply to more groups? For example, taking a group $G:=F\wr T$ where $|F|<\infty$ and $T$ is torsion free (possibly $\Z^n$ for some $n \in \N$) can one state that the base of $G$ is negligible in $G$?
\end{quR}
\begin{quR} Given a finitely generated group $H$, is the base of $G:=H\wr \Z$ negligible in $G$? Moreover, what if $\Z$ is replaced with another finitely generated infinite group?
\end{quR}
\bibliographystyle{amsalpha}
\def\cprime{$'$}
\providecommand{\bysame}{\leavevmode\hbox to3em{\hrulefill}\thinspace}
\providecommand{\MR}{\relax\ifhmode\unskip\space\fi MR }
% \MRhref is called by the amsart/book/proc definition of \MR.
\providecommand{\MRhref}[2]{%
  \href{http://www.ams.org/mathscinet-getitem?mr=#1}{#2}
}
\providecommand{\href}[2]{#2}
\end{document}